 \RequirePackage{fix-cm}

 \documentclass[smallextended]{svjour3} 
 \smartqed  

\usepackage{amssymb,amsmath,amsfonts,latexsym} 
\usepackage{hyperref}

\newtheorem{assumption}{Assumption}

\usepackage{mathtools}
\usepackage{multirow,array}
\usepackage{graphicx}
\usepackage{subfigure}
\usepackage{epstopdf}
\usepackage{float} 
\usepackage{color, xcolor}

\newcommand{\la}{\lambda}

\newcommand{\ppp}{\partial}

\newcommand{\www}{\widetilde}

\newcommand{\pppa}{\partial_t^{\alpha}}

\newcommand{\R}{\mathbb{R}}

\newcommand{\N}{\mathbb{N}}

\newcommand{\ooo}{\overline}
\newcommand{\OOO}{\Omega}

\newcommand{\xxee}{\xi,\eta}
\allowdisplaybreaks

\journalname{Fract. Calc. Appl. Anal.} 

\begin{document}


\title{Global and local existence of solutions for 
nonlinear systems of time-fractional diffusion 
equations}

\titlerunning{Global and local existence of solutions for 
nonlinear systems  \dots}

\author{
        Dian Feng$^1$ 
\and
        Masahiro Yamamoto$^{2,3}$ 
 }

\authorrunning{D. Feng \and M. Yamamoto} 

\institute{Dian Feng$^{1,*}$
\at
1. School of Mathematical Sciences, Fudan University, Shanghai 200433, China \\
\email{dfeng19@fudan.edu.cn} $^*$ corresponding author 
 \and
Masahiro Yamamoto$^{2,3}$
\at
2. Graduate School of Mathematical Sciences, The University of Tokyo, 3-8-1 Komaba, Meguro-ku, Tokyo 153-8914, Japan
\and
3. Department of Mathematics, Faculty of Science, Zonguldak B\"ulent Ecevit University, Zonguldak 67100, T\"urkiye\\
\email{myama@ms.u-tokyo.ac.jp}
}

\date{Received: XX 2024 / Revised: ... / Accepted: ......}


\maketitle

\begin{abstract}
In this paper, we consider initial-boundary value problems for two-component
nonlinear systems of time-fractional diffusion equations with 
the homogeneous Neumann boundary condition and non-negative initial values.
The main results are the existence of solutions global in time and the blow-up.
Our approach involves the truncation of the nonlinear terms, which enables us 
to handle all local Lipschitz continuous nonlinear terms, provided their sum is less than or 
equal to zero. 
By employing a comparison principle for the corresponding linear system, 
we establish also the non-negativity of the nonlinear system. 



\end{abstract} 


\section{Introduction} \label{sec:1}
\setcounter{section}{1} \setcounter{equation}{0} 

Let $\partial_t^{\alpha}$ be the fractional derivative of order 
$\alpha\in (0,1)$ defined on the fractional Sobolev spaces, which 
is defined as the closure of the classical Caputo derivative 
$$
d_t^{\alpha}v(t) := \frac{1}{\Gamma(1-\alpha)}
\int^t_0 (t-s)^{-\alpha}\frac{dv}{ds}(s) ds, \quad\mbox{ for }
v\in C^1[0,T] \mbox{ satisfying } v(0) = 0
$$  
(see Section \ref{Sec:well} for the details).  
Let $\Omega\subset \mathbb{R}^d, 
d=1,2,3$ be a bounded domain with smooth boundary $\partial \Omega$.
Moreover by $\nu$ we denote the unit outward normal vector to 
$\partial \Omega$ at $x\in \partial \Omega$. 

We consider an initial-boundary value problem 
for the following nonlinear system of time-fractional equations
with the homogeneous Neumann boundary condition
\begin{equation}\label{Eq:nonlin_system}
\left\{\begin{array}{l}
\partial_t^{\alpha} (u-a) = \Delta u + f(u,v),\quad x\in \Omega, 0< t < T, \\ 
\partial_t^{\alpha} (v-b) = \Delta v + g(u,v),\quad x\in \Omega, 0< t < T, \\
\partial_{\nu} u = \partial_{\nu} v = 0, \quad \text{ on } \partial 
\Omega\times (0,T),
\end{array}\right.
\end{equation}
where we assume some conditions on the nonlinear terms $f$ and $g$, as 
Assumption \ref{Ass:nonlin_term} in Section \ref{Sec:well} describes.

When $\alpha=1$, the system is reduced to a classical 
reaction-diffusion system,  and can be regarded as a special case of 
the Klausmeier-Gray-Scott model \cite{Kla,WS}. This system is considered to 
be a model equation for the vegetation pattern formation, which describes the 
self-organization of
vegetation spatial patterns resulting from the interaction between water source
distribution and plant growth.
In \cite{Pi} and \cite{Pi2}, it was proved that a weak solution in $L^1$ 
exists in the case where the coefficients in front 
of the spatial diffusion terms are not necessarily equal. 
However for the our fractional derivative case $0<\alpha<1$, we can not
apply the same technique as for $\alpha = 1$, so that we have to assume
that the coefficients of spatial diffusion terms are equal, that is,
$1$.  It should be a future work to discuss more general elliptic operators
in \eqref{Eq:nonlin_system}.

For $0<\alpha<1$, in contrast to the traditional reaction-diffusion process, we introduce a fractional order in the time derivative, indicating that the interaction between vegetation growth and water is influenced by soil medium heterogeneity. 
The fractional derivatives enable the representation of time memory effects, characterized by their nonlocal properties. 
Although there has been some research on various nonlinear fractional-order equation or coupled linear fractional-order equation systems, as seen in \cite{BP,LHL,FLY}, to the best of the authors' knowledge, there has been 
no investigation concerning nonlinear fractional-order systems.

Our approach involves the truncation of the nonlinear terms, which is 
inspired by Pierre \cite{Pi,Pi2}.
Through a comparison principle and a compact mapping property, 
we establish a convergent sequence of non-negative functions for the linear 
case. 
Moreover, $u+v$ has the energy estimation in the sense of $L^1$-norm, 
implying the convergence of the sequence in the $L^1$ sense. 
Ultimately, we obtain weak solutions in the $L^1$ norm.

The rest of this paper is organized as follows. In Section \ref{Sec:well}, we present the result regarding the global existence of weak solutions to the initial-boundary value problems for the nonlinear system of time-fractional diffusion equations.
Section \ref{Sec:lemma} is devoted to proving two key lemmas, which serve as the basis for the proofs of the main results. 
These lemmas establish the non-negativity of solution to 
the linear case, provided that the initial value and source term are non-negative, and 
demonstrate that the mapping from initial value and non-homogeneous term 
to solution of the corresponding linear system, is compact 
with suitable norms. 
Section \ref{Sec:thm_proof} completes the proof of the first main result. 
Section \ref{Sec:blow} discusses the blow-up of solutions under other conditions on the
nonlinear terms, which rejects the global 
existence in time, in general.  Section \ref{Sec:concl} provides some conclusions and remarks.

 \section{Well-posedness results} \label{Sec:well}
 \setcounter{section}{2} \setcounter{equation}{0} 
 
In this section, we deal with the following initial-boundary value problem for the nonlinear system of time-fractional equations \eqref{Eq:nonlin_system} with the time-fractional derivative of order $\alpha \in (0,1)$
\begin{equation}\label{Eq:nonlin_system_bdy}
\left\{\begin{array}{l}
\partial_t^{\alpha} (u-a) = \Delta u + f(u,v),\quad x\in \Omega,\ 0< t < T, 
                                    \\ 
\partial_t^{\alpha} (v-b) = \Delta v + g(u,v),\quad x\in \Omega,\ 0< t < T,\\
\partial_{\nu} u = \partial_{\nu} v = 0,\quad \text{ on } \partial \Omega\times (0,T),
\end{array}\right.
\end{equation}
along with the initial condition \eqref{Eq:initi} formulated below.

For $\alpha\in (0,1)$, let $\mathrm{d}^{\alpha}_t$ denote the classical Caputo derivative:
$$
\mathrm{d}^{\alpha}_t w(t):=\int_0^t \frac{(t-s)^{-\alpha}}{\Gamma(1-\alpha)} \frac{\mathrm{d}w}{\mathrm{d}s}(s) \mathrm{d}s,\ w\in W^{1,1}(0,T).
$$
Here $\Gamma(\cdot)$ denotes the gamma function. 
For a consistent treatment of nonlinear time-fractional diffusion equations, we extend the classical Caputo derivative $\mathrm{d}^{\alpha}_t$ as follows. First of all, we define the Sobolev-Slobodeckij space $H^{\alpha}(0,T)$ with the norm $\|\cdot\|_{H^{\alpha}(0,T)}$ for $0<\alpha<1$:
$$
\|w\|_{H^{\alpha}(0,T)}:=\left( \|w\|_{L^2(0,T)} + \int_0^T\int_0^T \frac{|w(t)-w(s)|^2}{|t-s|^{1+2\alpha}} \mathrm{d}t\mathrm{d}s \right)^{\frac{1}{2}}
$$
(e.g., Adams\cite{Ad}). Furthermore, we set $H^0(0,T):=L^2(0,T)$ and
$$
H_{\alpha}(0, T):=\left\{\begin{array}{ll}H^{\alpha}(0, T), & 0<\alpha<\frac{1}{2}, \\ \left\{w \in H^{\frac{1}{2}}(0, T) ; \int_{0}^{T} \frac{|w(t)|^{2}}{t} \mathrm{~d} t<\infty\right\}, & \alpha=\frac{1}{2}, \\ \left\{w \in H^{\alpha}(0, T) ; w(0)=0\right\}, & \frac{1}{2}<\alpha \leq 1\end{array}\right.
$$
with the norms defined by
$$
\|w\|_{H_{\alpha}(0, T)}:=\left\{\begin{array}{ll}\|w\|_{H^{\alpha}(0, T)}, & \alpha \neq \frac{1}{2}, \\ \left(\|w\|_{H^{\frac{1}{2}(0, T)}}^{2}+\int_{0}^{T} \frac{|w(t)|^{2}}{t} \mathrm{~d} t\right)^{\frac{1}{2}}, & \alpha=\frac{1}{2}.
\end{array}\right.
$$
Moreover, for $\beta>0$, we set
$$
J^{\beta}w(t):=\int^t_0 \frac{(t-s)^{\beta-1}}{\Gamma(\beta)}w(s)\mathrm{d}s,\ 0<t<T, w\in L^1(0,T).
$$
It was proved, for instance, in Gorenflo, Luchko, and Yamamoto\cite{GLY}, that the operator $J^{\alpha}: L^2(0,T)\to H_{\alpha}(0,T)$ is an isomorphism for $\alpha\in(0,1)$.

Now we are ready to give the definition of the extended Caputo derivative
$$
\partial_t^{\alpha}:=\left( J^{\alpha} \right)^{-1},\ \mathcal{D}(\partial_t^{\alpha})=H_{\alpha}(0,T).
$$
Henceforth $\mathcal{D}(\cdot)$ represents the domain of an operator under consideration.
It has been demonstrated that $\partial_t^{\alpha}$ represents the minimal closed extension of $\mathrm{d}^{\alpha}_t$, where $\mathcal{D}(\mathrm{d}^{\alpha}_t):=\{ v\in C^1[0,T];v(0)=0 \}$. The relation $\partial_t^{\alpha} w = \mathrm{d}^{\alpha}_t w$ is maintained for $w\in C^1[0,T]$ with $w(0)=0$. 
For further details, we can refer to Gorenflo et al. \cite{GLY} 
and Yamamoto\cite{Y4}.

Now we will define initial condition for problem \eqref{Eq:nonlin_system_bdy} 
as follows:
\begin{equation}\label{Eq:initi}
u(x,\cdot)-a(x) \in H_{\alpha}(0,T), v(x,\cdot)-b(x), \in H_{\alpha}(0,T)\ \text{ for almost all } x\in \Omega
\end{equation}
and write down a complete formulation of an initial-boundary value problem for the nonlinear system of time-fractional equations \eqref{Eq:nonlin_system}:
\begin{equation}\label{Eq:nonlin_system_formu}
\left\{\begin{array}{l}
\partial_t^{\alpha} (u-a) = \Delta u +f(u,v),\quad x\in \Omega, 0< t < T, \\ 
\partial_t^{\alpha} (v-b) = \Delta v +g(u,v),\quad x\in \Omega, 0< t < T,\\
\partial_{\nu} u = \partial_{\nu} v = 0,\quad \text{ on } \partial \Omega\times (0,T),\\
u(x,\cdot)-a(x), v(x,\cdot)-b(x) \in H_{\alpha}(0,T),\quad \text{ for almost all } x\in \Omega.
\end{array}\right.
\end{equation}
It is worth mentioning that the terms $\partial_t^{\alpha} (u-a)$ and $\partial_t^{\alpha} (v-b)$ in the first two lines of \eqref{Eq:nonlin_system_formu} 
are well-defined for almost all $x \in \Omega$ and $0<t<T$, 
due to the inclusion 
formulated in the last line of \eqref{Eq:nonlin_system_formu}. Especially for $\frac{1}{2} < \alpha < 1,$ noting that $w\in H_{\alpha}(0,T)$ implies $w(0)=0$ by the trace theorem, we can understand that the left-hand side means that $u(x,0)=a(x)$ and $v(x,0)=b(x)$ in the trace sense with respect to $t$. While for 
$\alpha<\frac{1}{2}$, we do not need any initial conditions.

Now we are well prepared to investigate the initial-boundary value problem \eqref{Eq:nonlin_system_formu}. We first provide the definition of weak solutions for equations \eqref{Eq:nonlin_system_formu}.
Henceforth we set 
$$
Q:= \OOO \times (0,T)
$$
and
$$
(\partial_t^{\alpha})^*\psi(x,s):= \frac{-1}{\Gamma(1-\alpha)}
\int^T_s (t-s)^{-\alpha} \frac{\partial \psi}{\partial t}(x,t) 
 \mathrm{d}t
$$
for $\psi \in C^1([0,T];L^1(\OOO))$.
Then: 
\begin{definition}\label{Def:weak}
We call $(u,v) \in L^1(0,T;L^1(\Omega))^2$ a weak solution to 
\eqref{Eq:nonlin_system_formu} if  
\begin{equation}\label{Eq:weak_solution}
\left\{\begin{array}{l}
\int_Q (u-a)(\partial_t^{\alpha})^*\psi \,
\mathrm{d}x \mathrm{d}t 
= \int_Q  u \Delta \psi \, \mathrm{d}x \mathrm{d}t
  + \int_Q  f(u,v) \psi \, \mathrm{d}x \mathrm{d}t,          \cr\\ 
\int_{Q} (v-b)(\partial_t^{\alpha})^*\psi \, \mathrm{d}x  \mathrm{d}t
= \int_Q  v \Delta \psi \, \mathrm{d}x\mathrm{d}t 
+ \int_Q  g(u,v) \psi \, \mathrm{d}x \mathrm{d}t,
\end{array}\right.
\end{equation}
for all $\psi(x,t)\in C^{\infty}(\ooo{Q})$ satisfying
$\partial_{\nu} \psi|_{\partial\Omega\times(0,T)}=0$ and 
$\psi(\cdot, T)=0$ in $\Omega$.
\end{definition}
If a weak solution $(u, v) \in L^1(0,T;L^1(\Omega))^2$ is sufficiently smooth,
then we can verify by the definition that 
$(u,v)$ satisfies \eqref{Eq:nonlin_system} pointwise in the classical sense.

Based on the definition, we can give a lemma immediately.
\begin{lemma}
If \eqref{Eq:nonlin_system_formu} has a solution $(u,v) 
\in L^2(0,T;H^2(\OOO))^2$ such that $u-a, v-b 
\in H_{\alpha}(0,T;L^2(\OOO))$, then $(u,v)$ is a weak solution.
\end{lemma}
\begin{proof}
Since $(u,v)$ is a solution of \eqref{Eq:nonlin_system_formu}, we have
\begin{equation}\label{Eq:strong}
\left\{\begin{array}{l}
\partial_t^{\alpha} (u-a) = \Delta u +f(u,v), \\ 
\partial_t^{\alpha} (v-b) = \Delta v +g(u,v).
\end{array}\right.
\end{equation}
Multiply both sides of \eqref{Eq:strong} by $\psi(x,t)\in C^{\infty}(\ooo{Q})$, taking into account the boundary condition of $\psi(x,t)$ in Definition \ref{Def:weak}, and integrate by parts. This leads us to the same formula as \eqref{Eq:weak_solution}, thus completing the proof.
\end{proof}

\begin{assumption}\label{Ass:nonlin_term}
Let $f,g: \mathbb{R}^2 \longrightarrow \R$ be local Lipschitz continuous.
More precisely, for arbitrarily given $M>0$, there exists a constant 
$C=C_M>0$ such that
\begin{equation}\label{Eq:lips_condi}
|f(\xi,\eta)-f(\widehat{\xi},\widehat{\eta})|
+ |g(\xi,\eta)-g(\widehat{\xi},\widehat{\eta})|
\leq C (|\xi-\widehat{\xi}|+|\eta-\widehat{\eta}|).
\end{equation}
for all $\xi, \eta, \widehat{\xi},\widehat{\eta}\in [-M,M]$. 
Moreover we assume  
$$
f(0,\eta) = g(\xi,0) = 0, \quad \mbox{ for all } \xi\ge 0 \mbox{ and }
\eta \ge 0.
$$
Finally we assume that we can find a constant $\lambda>0$
such that 
$$
f(\xi,\eta)+\lambda g(\xi,\eta) \leq 0 \quad \mbox{ for all } \xi, \eta \ge 0.
$$
\end{assumption}

Here, compared to the conditions required in \cite{Pi}, we impose stronger 
requirements on $f$ and $g$, which are crucial 
for the convergence of the truncated non-linear terms in Sections 
\ref{Sec:thm_proof}. Unlike \cite{Pi}, we no longer restrict $\lambda$ to be less 
than or equal to 1. 

Now we state the first main result in this article, which validates the 
well-posedness for given $T>0$ of the initial-boundary value problem 
\eqref{Eq:nonlin_system_formu}.

\begin{theorem}\label{Th:well-pose}
In \eqref{Eq:nonlin_system_formu}, we assume $a,b\geq 0$, $a,b\in H^1(\Omega) \cap L^{\infty}(\OOO)$, with $f,g$ satisfy Assumption \ref{Ass:nonlin_term}. Then, for arbitrarily given $T>0$, 
there exists at least one weak solution 
$(u,v) \in L^1(\Omega\times (0,T))^2$ such that $u, v \ge 0$ in $Q$.
\end{theorem}
We do not know the uniqueness of weak solution under our assumption, which is 
the same as for the case $\alpha=1$.
 
Before starting with the proof of Theorem \ref{Th:well-pose}, we introduce some notations and derive several results necessary for the proof.

For an arbitrary constant $p_0>0$, define an elliptic operator $A$ 
as follows:
\begin{equation}
\left\{\begin{array}{l}\left(-A v\right)(x):=\Delta v(x)-p_{0} v(x), \quad x \in \Omega, \\ \mathcal{D}(A)=\left\{v \in H^{2}(\Omega); \partial_{\nu} v=0 \text { on } \partial \Omega\right\}.
\end{array}\right.
\end{equation}
Henceforth, by $\|\cdot\|$ and $(\cdot,\cdot)$ we denote the standard norm and the scalar product in $L^2(\Omega)$, respectively. It is well-konwn that the operator $A$ is self-adjoint in $L^2(\Omega)$. Moreover, for a sufficiently large constant $p_0>0$, we can verify that $A$ is positive definite\cite{LY1}. Therefore, by choosing a constant $p_0$ large enough, the spectrum of the operator $A$ is comprised of discrete positive eigenvalues, herein represented as $0 < \lambda_1 \leq \lambda_2 \leq \cdots$, each uniquely designated by its multiplicity. Additionally, $\lambda_n \to \infty$ as $n\to \infty$. Set $\varphi_n$ be the eigenfunction corresponding to the eigenvalue $\lambda_n$ such that $A \varphi_n=\lambda_n \varphi_n$, and $(\varphi_i, \varphi_j)=\delta_{ij}$. Then the sequence $\{\varphi_n\}_{n\in \mathbb{N}}$ is orthonormal basis in $L^2(\Omega)$.
For any $\gamma >0$, we can define the fractional power $A^{\gamma}$ of the operator $A$ by the following relation (see, e.g.,\cite{Pa}): 
$$
A^{\gamma}v=\sum_{n=1}^{\infty} \lambda^{\gamma}_{n} (v,\varphi_n)\varphi_n,
$$
where
$$
v\in \mathcal{D}(A^{\gamma}):=\left\{ v\in L^2(\Omega): \sum_{n=1}^{\infty} \lambda_n^{2\gamma} (v,\varphi_n)^2 < \infty \right\}
$$
and
$$
\|A^{\gamma}v\|=\left( \sum_{n=1}^{\infty} \lambda_n^{2\gamma} (v,\varphi_n)^2  \right)^{\frac{1}{2}}.
$$
We have $\mathcal{D}(A^{\gamma}) \subset H^{2\gamma}(\Omega)$ for $\gamma>0$. Since $\mathcal{D}(A^{\gamma}) \subset L^2(\Omega)$, identifying the dual $(L^2(\Omega))^{\prime}$ with itself, we have $\mathcal{D}(A^{\gamma}) \subset L^2(\Omega) \subset \mathcal{D}(A^{\gamma})^{\prime}$. Henceforth we set $\mathcal{D}(A^{-\gamma}) = \mathcal{D}(A^{\gamma})^{\prime}$, which consists of bounded linear functionals on $\mathcal{D}(A^{\gamma})$. For $w\in \mathcal{D}(A^{-\gamma})$ and $v\in \mathcal{D}(A^{\gamma})$, by ${}_{-\gamma}\langle w,v \rangle_{\gamma}$, we denote the value which is obtained by operating $w$ to $v$. We note that $\mathcal{D}(A^{-\gamma})$ is a Hilbert space with norm:
$$
\|w\|_{\mathcal{D}(A^{-\gamma})} = \left\{ \sum_{n=1}^{\infty} \lambda^{-2\gamma}_n \left|{}_{-\gamma}\langle w,v \rangle_{\gamma}\right|^2 \right\}^{\frac{1}{2}}.
$$
We further note that
$$
{}_{-\gamma}\langle w,v \rangle_{\gamma} = (w,v) \text{ if } w\in L^2(\Omega) \text{ and } v\in \mathcal{D}(A^{\gamma}).
$$

Moreover we define two other operators $S(t)$ and $K(t)$ as:
$$
S(t)a=\sum_{n=1}^{\infty} E_{\alpha,1}(-\lambda_n t^{\alpha})(a,\varphi_n)\varphi_n,\  a\in L^2(\Omega), t>0
$$
and
$$
K(t)a = \sum_{n=1}^{\infty}t^{\alpha-1}E_{\alpha,\alpha}(-\lambda_n t^{\alpha})(a,\varphi_n)\varphi_n,\  a\in L^2(\Omega), t>0,
$$
where $E_{\alpha,\beta}(z)$ denotes the Mittag-Leffler function defined by a convergent series as follows:
$$
E_{\alpha,\beta}(z)=\sum_{k=0}^{\infty} \frac{z^k}{\Gamma(\alpha k+\beta)},\ \alpha>0, \beta\in \mathbb{C}, z\in \mathbb{C}.
$$
It follows directly from the definitions given above that $A^{\gamma}K(t)a=K(t)A^{\gamma}a$ and $A^{\gamma}S(t)a=S(t)A^{\gamma}a$ for $a\in \mathcal{D}(A^{\gamma})$. Moreover, the following norm estimates were proved in\cite{SY}:
\begin{equation}\label{Equ:ken_norm}
\begin{aligned}
\|A^{\gamma}S(t)a\|&\leq Ct^{-\alpha\gamma}\|a\|,
\\
\|A^{\gamma}K(t)a\| &\leq Ct^{\alpha(1-\gamma)-1} \|a\|,\ a\in L^2(\Omega), t>0, 0\leq \gamma \leq 1.
\end{aligned}
\end{equation}

\section{Key Lemma} \label{Sec:lemma}

\setcounter{section}{3} \setcounter{equation}{0} 

For the proof of Theorem \ref{Th:well-pose}, we show several lemmas.

\begin{lemma} (Non-negativity) \label{Lm:non_nega}
We assume that $c\in L^{\infty}(Q)$, $a \in H^1(\OOO)$, $a \ge 0$ in 
$\OOO$ and $F \in L^2(Q)$, $F\ge 0$ in $Q$.
Let $u \in L^2(0,T;H^2(\OOO))$ satisfy \eqref{Eq:lin_system} and 
$u-a \in H_{\alpha}(0,T;L^2(\OOO))$:
\begin{equation}\label{Eq:lin_system}
\left\{\begin{array}{l}
\partial_t^{\alpha} (u-a) = \Delta u + c(x,t)u + F(x,t),\quad x\in \Omega, 
0< t < T, \\
\partial_{\nu} u = 0,\quad \text{ on } \partial \Omega, 0<t<T,    \\
u(x,\cdot)-a(x)\in  H_{\alpha}(0,T),\quad  \text{ for almost all } 
x\in \Omega.
\end{array}\right.
\end{equation}
Then $u \ge 0$ in $Q$.
\end{lemma}

The proof is simialr to Luchko and Yamamoto \cite{LY,LY1} where 
the coefficient $c$ is assumed to be more smooth than $L^{\infty}(Q)$.
For the case $c\in L^{\infty}(Q)$, approximating $c$ by 
functions in $C^{\infty}_0(Q)$ in the space
$L^{\kappa}(Q)$ with arbitrarily fixed $\kappa>1$.
For completeness, we provide the proof in Appendix.

Here and henceforth we set $A=-\Delta + p_0$. 
For $w_0 \in H^1(\OOO)$ and $F \in L^2(Q)$, we consider 
\begin{equation}\label{Eq:frac_single}
\left\{\begin{array}{l}\partial_{t}^{\alpha}\left(w-w_{0}\right)+A w
=F(x, t),\quad x \in \Omega, 0<t<T, \\ \partial_{\nu} w=0,\quad x \in \partial \Omega, 0<t<T, \\ w(x, \cdot)-w_{0}(x) \in H_{\alpha}(0, T),\quad \text { for almost all } 
x \in \Omega.
\end{array}\right.
\end{equation}
Then we know (e.g., \cite{KRY}) that there exists a unique solution
$w \in L^2(0,T;H^2(\OOO))$ to \eqref{Eq:frac_single} satisfying 
$w - w_0 \in H_{\alpha}(0,T;L^2(\OOO))$.
By $w(w_0,F)$ we denote the solution.

We can prove a compactness property:
\begin{lemma}\label{Lm:comp}
Let $M_1>0$ be arbitrarily chosen constant.
Then a set
$$
\left\{ w(w_0^n, F_n)| (w_0^n, F_n) \in H^1(\OOO)\times
L^2(Q),\ \Vert w_0^n\Vert_{L^1(\OOO)}
+ \Vert F_n\Vert_{L^1(Q)} \le M_1\right\}
$$
contains a subsequence which converges in $L^1(Q)$.
\end{lemma}

\begin{proof}
We introduce an operator $\tau_h(f)$ for $h>0$ and 
$f \in L^1(0,T;X)$ with Banach space $X$ by 
$\tau_h(f)(t)=f(t+h)$.

We show two lemmas for the proof of Lemma \ref{Lm:comp}.
\begin{lemma} (Simon\cite{Si}) \label{Lm:Simon_compact}
Let $X\subset B \subset Y$ be Banach spaces 
and the embedding $X\to Y$ be compact.  We assume  
$$
\|v\|_{B}\leq C \|v\|^{1-\theta}_{X} \|v\|^{\theta}_{Y}, \quad \forall v \in X\cap Y, \quad \text{ where } 0\leq \theta \leq 1,
$$
$$
\mbox{ a set } U \mbox{ is bounded in } L^{1}(0,T;X)
$$
and
$$
\|\tau_h u-u \|_{L^{1}(0,T-h;Y)}\to 0 \mbox{ as } h\to 0 \mbox{ uniformly for }
u \in U.
$$
Then $U$ is relatively compact in $L^{1}(0,T;B)$.
\end{lemma}

\begin{lemma}\label{Le:frac_ope_est}
Let $\Omega \in \mathbb{R}^n$, if $\gamma > \frac{d}{4}$, then we have
$$
\|A^{-\gamma} w\|_{L^2(\Omega)} \leq C \|w\|_{L^1(\Omega)}.
$$ 
\end{lemma}
{\it Proof of Lemma \ref{Le:frac_ope_est}}

Since $\mathcal{D}(A^{\gamma})\subset H^{2\gamma}(\Omega)\subset L^{\infty}(\Omega)$, we can get $\|v\|_{L^{\infty}(\Omega)}\leq C\|A^{\gamma}v\|_{L^2(\Omega)}$.
Moreover, we have
$$
\begin{aligned}
\|A^{-\gamma} w \|_{L^2(\Omega)} &= \sup_{\|\varphi\|_{L^2(\Omega)}=1} |(A^{-\gamma}w, \varphi)_{L^2(\Omega)}| \\
&= \sup_{\|\varphi\|_{L^2(\Omega)}=1}  |(w, A^{-\gamma} \varphi)_{L^2(\Omega)}| \\
&\leq C \|w\|_{L^1(\Omega)} \|A^{-\gamma} \varphi\|_{L^{\infty}(\Omega)} \\
&\leq C \|w\|_{L^1(\Omega)} \|A^{\gamma}A^{-\gamma} \varphi\|_{L^{2}(\Omega)} \\
&\leq C \|w\|_{L^1(\Omega)}.
\end{aligned}
$$
Then we finished the proof.
$\blacksquare$

Now we proceed to the proof of Lemma \ref{Lm:comp}.
In Lemma \ref{Lm:Simon_compact}, set $X=\mathcal{D}(A^{1-\delta}), Y=\mathcal{D}(A^{-\delta})$ and $B=L^1(\Omega)$. The intermediate spaces \cite{BL} 
between $X$ and $Y$ are given by
$$[X,Y]_{\theta}=\mathcal{D}(A^{1-\delta-\theta}),\quad \forall \theta\in [0,1].
$$ 
If we choose $\theta=1-\delta$, we obtain $[X,Y]_{1-\delta}=\mathcal{D}(A^0)=L^2(\Omega)$.
Referring to the interpolation result of Sobolev spaces, this implies
$$
\|w\|_{L^2(\Omega)} \leq C\|w\|_{\mathcal{D}(A^{1-\delta})}^{\delta} \|w\|_{\mathcal{D}(A^{-\delta})}^{1-\delta}.
$$
Applying the embedding theory of $L^p(\Omega)$ for $p\in [1,\infty]$, we have $\|w\|_{L^1(\Omega)}\leq C \|w\|_{L^2(\Omega)}$, which means 
$$
\|w\|_{L^1(\Omega)} \leq C\|w\|_{\mathcal{D}(A^{1-\delta})}^{\delta} \|w\|_{\mathcal{D}(A^{-\delta})}^{1-\delta}.
$$
In order to apply Lemma \ref{Lm:Simon_compact}, we still need to prove the following inclusion relation:
\begin{equation}\label{Eq:com_embed}
\mathcal{D}(A^{1-\delta}) \subset L^1(\Omega) \subset \mathcal{D}(A^{-\delta}).
\end{equation}
To establish \eqref{Eq:com_embed}, we start by considering the following relationship:
$$\mathcal{D}(A^{\delta})\subset H^{2\delta}(\Omega)\subset H^{2\delta^{\prime}}(\Omega)\subset L^{\infty}(\Omega),$$
where $\frac{d}{4}<\delta^{\prime}<\delta$. 
The Kondrachov embedding theorem implies 
$H^{2\delta}(\Omega)\subset H^{2\delta^{\prime}}(\Omega)$ is compact. 
Consequently, we obtain that $\mathcal{D}(A^{\delta})\subset L^{\infty}(\Omega)$ is compact. By the definition, $L^1(\Omega)\subset \left( 
L^{\infty}(\Omega) \right)^{\prime} \subset \left( 
\mathcal{D}(A^{\delta}) \right)^{\prime}=\mathcal{D}(A^{-\delta})$. 
Here $X^{\prime}$ denotes the dual space of a Banach space $X$.
Hence, we have proved \eqref{Eq:com_embed}, and the mapping $\mathcal{D}(A^{1-\delta}) \to \mathcal{D}(A^{-\delta})$ is compact.

We set $w_n:= w(w_0^n, F_n)$ for each $n\in \N$.
In order to prove Lemma \ref{Lm:comp}, in terms of Lemma \ref{Lm:Simon_compact},
it is sufficient to show that 
$\|w_n\|_{L^1(0,T;\mathcal{D}(A^{1-\delta}))}$ is bounded and 
$\|\tau_h w_n- w_n\|_{L^1(0,T;\mathcal{D}(A^{-\delta}))}\to 0 $  
as $h\to 0$ uniformly for each $n \in \N$. 

Using the operator $S(t)$ and $K(t)$, the solution can be expressed 
by the formula\cite{LY1,LY2}:
\begin{equation}\label{Equ:kernel_expre}
w_n(t)=S(t)w_0^n + \int_0^t K(t-s)F_n(s) \mathrm{d}s.
\end{equation}
Writing \eqref{Equ:kernel_expre} as
$$
w_n(t) =S(t)w_0^n+ \int_0^t A^{\gamma} K(t-s) A^{-\gamma} F_n(s) \mathrm{d}s.
$$
Then, multiplying both sides by $A^{1-\delta}$, we obtain
\begin{equation}\label{Eq:A_1_delta}
A^{1-\delta}w_n(t) =AS(t)A^{-\delta}w_0^n 
+ \int_0^t A^{1-\delta+\gamma} K(t-s) A^{-\gamma} F_n(s) \mathrm{d}s,
\end{equation}
where $\delta > \gamma> \frac{d}{4}$.
Applying the norm estimates \eqref{Equ:ken_norm}, we have $\|A^{1-\delta+\gamma} K(t-s)\|\leq (t-s)^{\alpha(\delta-\gamma)-1}$ and $\|A S(t)\|\leq t^{-\alpha}$.
By combining this with Lemma \ref{Le:frac_ope_est}, we can derive
$$
\left \|A^{1-\delta} w_n(t)\right\|_{L^2(\Omega)} 
\leq C\int_0^t s^{-\alpha} \|w_0^n\|_{L^1(\Omega)} \mathrm{d}s  
+ C \int_0^t (t-s)^{\alpha(\delta-\gamma)-1} \|F_n(s)\|_{L^1(\Omega)} 
\mathrm{d}s. 
$$
Furthermore, applying Young's inequality, we reach 
$$
\left \|A^{1-\delta} w_n\right\|_{L^1(0,T;L^2(\Omega))} 
\leq C \left(\|w_0^n\|_{L^1(\Omega)} + \|F_n\|_{L^1(0,T;L^1(\Omega))} \right).
$$
Multiply \eqref{Eq:frac_single} by $A^{-\delta}$, and we have
$$
A^{-\delta} \partial_t^{\alpha} (w_n-w_0^n) = -A^{1-\delta} w_n 
+ A^{-\delta} F_n.
$$
Applying \eqref{Eq:A_1_delta} and Lemma \ref{Le:frac_ope_est} again, we can 
obtain 
$$
\begin{aligned}
&\|A^{-\delta}\partial_t^{\alpha} (w_n-w_0^n)\|_{L^1(0,T;L^2(\Omega))} \\
\leq & C\left(\|w_0^n\|_{L^1(\Omega)} + \|F_n\|_{L^1(0,T;L^1(\Omega))} \right)+  \|A^{-\delta}F_n\|_{L^1(0,T;L^2(\Omega))}\\
\leq &C\left(\| w_0^n\|_{L^1(\Omega)} + \|F_n\|
_{L^1(0,T;L^1(\Omega))} \right).
\end{aligned}
$$
Here we have obtained
\begin{equation}
\begin{aligned}
\|A^{1-\delta}w_n\|_{L^1(0,T;L^2(\Omega))}& \leq C \left(\|w_0^n\|
_{L^1(\Omega)} + \|F_n\|_{L^1(0,T;L^1(\Omega))} \right),  \\
\|A^{-\delta}\partial_t^{\alpha} (w_n-w_0^n)\|_{L^1(0,T;L^2(\Omega))}& \leq C \left(\|w_0^n\|_{L^1(\Omega)} + \|F_n\|_{L^1(0,T;L^1(\Omega))} \right). 
\end{aligned}
\end{equation}

To utilize the Lemma \ref{Lm:Simon_compact}, we now need to prove:
$$
\|w_n(t+h)-w_n(t)\|_{L^1(0,T-h;\mathcal{D}(A^{-\delta}))} \longrightarrow 0
\quad \mbox{ as } h\to 0.
$$
Denote $\widetilde{w}_n:= \partial_t^{\alpha}(A^{-\delta}(w_n-w_0^n))
\in L^{1}(0,T;L^2(\Omega))$. 
Applying the fractional integral $J^{\alpha}$ to both sides, we have 
$\widetilde{w}_n:= A^{-\delta}(w_n-w_0^n) = J^{\alpha} \widetilde{w}_n$. 
Noticing that $w_n(t+h)-w_n(t)=(w_n(t+h)-w_0^n)-(w_n(t)-w_0^n)$, 
and 
$$
\|w_n(t+h)-w_n(t)\|_{L^1(0,T-h;\mathcal{D}(A^{-\delta}))}  
=  \|J^{\alpha} \widetilde{w}_n(t+h) - J^{\alpha} \widetilde{w}_n(t)\|
_{L^1(0,T-h;L^2(\Omega))}.
$$
Now, we represent $J^{\alpha} \widetilde{w}_n(t+h) 
- J^{\alpha} \widetilde{w}_n(t)$:
\begin{equation}
\begin{aligned}
& J^{\alpha} \widetilde{w}_n(t+h) - J^{\alpha} \widetilde{w}_n(t)\\
= & \frac{1}{\Gamma(\alpha)} \int_0^t (t+h-s)^{\alpha-1} \widetilde{w}_n(s)
\mathrm{d}s - \frac{1}{\Gamma(\alpha)} \int_0^t (t-s)^{\alpha-1} 
\widetilde{w}_n(s)\mathrm{d}s \\
+& \int_t^{t+h} (t+h-s)^{\alpha-1} \widetilde{w}_n(s) \mathrm{d} s.
\end{aligned}
\end{equation}
Let 
$$
I(t):= \int_0^t ( (t+h-s)^{\alpha-1}-(t-s)^{\alpha-1}) \widetilde{w}_n(s)
\mathrm{d}s, \quad 
J(t):=\int_t^{t+h} (t+h-s)^{\alpha-1} \widetilde{w}_n(s) \mathrm{d} s.
$$
As for the first term concerning $I(t)$, we have
$$
\|I(t)\|_{\mathcal{D}(A^{-\delta})} \leq \int_0^t \left((t-s)^{\alpha-1} - (t+h-s)^{\alpha-1}\right) \| \widetilde{w}_n(s)\|_{\mathcal{D}(A^{-\delta})} 
\mathrm{d}s.
$$
Integrating both sides with respect to the time variable $t$, we obtain the following estimate:
\begin{equation}
\begin{aligned}
&\int_0^{T-h} \|I(t)\|_{\mathcal{D}(A^{-\delta})} \mathrm{d}t \\
\leq &C \int_0^{T-h}( (T-h-s)^{\alpha} - (T-s)^{\alpha} + h^{\alpha} ) 
\|\widetilde{w}_n(s)\|_{\mathcal{D}(A^{-\delta})} \mathrm{d}s \\
\leq& C \int_0^{T-h} h^{\alpha}  \|\widetilde{w}_n(s)\|_{\mathcal{D}(A^{-\delta})} \mathrm{d}s \\
\leq& C h^{\alpha} \|\widetilde{w}_n\|_{L^1(0,T;\mathcal{D}(A^{-\delta}))}.
\end{aligned}
\end{equation}
After that, we start the estimate of $J(t)$:
\begin{equation}
\begin{aligned}
& \int^{T-h}_0 \|J(t)\|_{\mathcal{D}(A^{-\delta})} \mathrm{d}t \\
\leq &\int^{T-h}_0 \int_t^{t+h} (t+h-s)^{\alpha-1} \|\widetilde{w}_n(s)
\|_{\mathcal{D}(A^{-\delta})} \mathrm{d}s \mathrm{d}t.
\end{aligned}
\end{equation}
Let the integrand $(t+h-s)^{\alpha-1} \|\widetilde{w}_n(s)
\|_{\mathcal{D}(A^{-\delta})}$ be denoted as $b(s,t)$.
Then, by Fubini's theorem, we can exchange the orders of integration:
$$
\int^{T-h}_0 \int_t^{t+h} b(s,t)\mathrm{d}s\mathrm{d}t = \int_0^h \int_0^s b(s,t) \mathrm{d}t\mathrm{d}s+ \int_h^{T-h} \int_{s-h}^s b(s,t) \mathrm{d}t\mathrm{d}s + \int_{T-h}^T \int_{s-h}^{T-h} b(s,t) \mathrm{d}t\mathrm{d}s.
$$
Then the estimate for $J(t)$ can be divided into the following three parts:
\begin{equation}
\begin{aligned}
J_1 &=\int_0^{h} \int_{0}^s (t+h-s)^{\alpha-1} \|\widetilde{w}_n(s)\|
_{\mathcal{D}(A^{-\delta})}   \mathrm{d}t \mathrm{d}s \\
& = \int_0^{h} \|\widetilde{w}_n(s)\|_{\mathcal{D}(A^{-\delta})} 
\left( \int_{0}^s    (t+h-s)^{\alpha-1}   \mathrm{d}t \right) \mathrm{d}s \\
&\leq C h^{\alpha} \|\widetilde{w}_n\|_{L^1((0,T);\mathcal{D}(A^{-\delta}))}
\end{aligned}
\end{equation}
and 
\begin{equation}
\begin{aligned}
J_2 &=\int_h^{T-h} \int_{s-h}^s (t+h-s)^{\alpha-1} \|\widetilde{w}(s)
\|_{\mathcal{D}(A^{-\delta})}   \mathrm{d}t \mathrm{d}s \\
& = \int_h^{T-h} \|\widetilde{w}_n(s)\|_{\mathcal{D}(A^{-\delta})} 
\left( \int_{s-h}^s    (t+h-s)^{\alpha-1}   \mathrm{d}t \right) \mathrm{d}s \\
&\leq C h^{\alpha} \|\widetilde{w}_n\|_{L^1((0,T);\mathcal{D}(A^{-\delta}))}
\end{aligned}
\end{equation}
and
\begin{equation}
\begin{aligned}
J_3 &=\int^T_{T-h} \int_{s-h}^{T-h} (t+h-s)^{\alpha-1} \|\widetilde{w}_n(s)
\|_{\mathcal{D}(A^{-\delta})}   \mathrm{d}t \mathrm{d}s \\
& = \int_{T-h}^{T} \|\widetilde{w}_n(s)\|_{\mathcal{D}(A^{-\delta})} 
\left( \int_{s-h}^{T-h}    (t+h-s)^{\alpha-1}   
\mathrm{d}t \right) \mathrm{d}s \\
&\leq C h^{\alpha} \|\widetilde{w}_n\|_{L^1((0,T);\mathcal{D}(A^{-\delta}))}.
\end{aligned}
\end{equation}
Consequently, 
$$
\|\tau_h w_n-w_n\|_{L^1(0,T;\mathcal{D}(A^{-\delta}))} \leq O(h^{\alpha}).
$$
Thus we have completed the proof of Lemma \ref{Lm:comp}.
\end{proof}

\section{Completion of the proof of Theorem \ref{Th:well-pose}}\label{Sec:thm_proof}
\setcounter{section}{4} \setcounter{equation}{0} 

\begin{proof}
Henceforth $C>0, C_1 > 0$, etc. denote generic constants which are
independent of $n\in \N$ and the choice of variables $\xi, \eta$.
Henceforth when we write $C=C_0(n)$, etc., we understand that the constant 
depends on $n$ such as $C_0(n)$ above.

We begin by constructing approximating solutions $u_n$ and $v_n$. 
In order to handle 
the nonlinear terms $f$ and $g$, we introduce the truncated functions 
below.  We choose $\psi_1 \in C_0^{\infty}(\R^2)$ such that  
$0\le \psi_1(\xi,\eta) \le 1$ for all $\xi, \eta \in \mathbb{R}$ and
$$
\psi_1(\xi,\eta)=
\left\{\begin{array}{l}
1, \quad \mbox{ if } |\xi|,|\eta|\leq 1, \\ 
0, \quad \mbox{ if }  |\xi|\ge 2 \mbox{ or } |\eta|\geq 2.
\end{array}\right.
$$
Next, we set 
$\psi_n(\xi,\eta)=\psi_1\left(\frac{\xi}{n}, \frac{\eta}{n}\right)$. 
With this choice, $0\leq \psi_n\leq 1$ for all $n\in \N$, and $\psi_n$ tends 
pointwise to 1 as $n$ tends to $\infty$.

We define truncated nonlinear terms as follows
$$
\left\{\begin{array}{l}
f_n(\xi,\eta)= \psi_n(\xi,\eta) f(\xi,\eta), \quad \xi, \eta \in \R,  \\ 
g_n(\xi,\eta)= \psi_n(\xi,\eta) g(\xi,\eta), \quad \xi, \eta \in \R.
\end{array}\right.
$$
Then we can verify that there exists a constant $C_0(n) > 0$ such that 
\begin{equation}\label{Eq:fngn_norm_infty}
\left\{ \begin{array}{l}
\Vert f_n\Vert_{L^{\infty}(\R^2)}, \, \Vert g_n\Vert_{L^{\infty}(\R^2)} \le 
C_0(n), \\
\Vert \ppp_k f_n\Vert_{L^{\infty}(\R^2)}, \,  
  \Vert \ppp_k g_n\Vert_{L^{\infty}(\R^2)} \le C_0(n), \quad k=1,2
\end{array}\right.
\end{equation}
and
\begin{equation}\label{Eq:fngn_00}
f_n(0,\eta) = g_n(\xi,0) = 0 \quad \mbox{ for all } \xi, \eta \in \R.
\end{equation}
Indeed, \eqref{Eq:fngn_00} is verified directly.  Moreover
$$
\vert f_n(r,s)\vert + \vert g_n(r,s)\vert 
\le \sup_{\vert \xi\vert, \vert \eta\vert \le 2n}
\vert f(\xi,\eta)\vert 
+ \sup_{\vert \xi\vert, \vert \eta\vert \le 2n}
\vert g(\xi,\eta)\vert,
$$
for all $r,s\in \R$, which means the first estimate in \eqref{Eq:fngn_norm_infty}.
The rest estimates in \eqref{Eq:fngn_norm_infty} follow from the definition of 
$f_n, g_n$. Furthermore we note that $f_n, g_n$ satisfy Assumption \ref{Ass:nonlin_term}.

A truncated form of equation \eqref{Eq:nonlin_system_formu} is given by:
\begin{equation}\label{Eq:nonlin_system_appro}
\left\{\begin{array}{l}
\partial_t^{\alpha} (u_n-a) = -Au_n + F_n(u_n(x,t), v_n(x,t)),\quad x\in \Omega, 0< t < T,  \\ 
\partial_t^{\alpha}(v_n-b) = -Av_n + G_n(u_n(x,t),v_n(x,t)),\quad x\in \Omega, 0< t < T,  \\
\partial_{\nu} u_n = \partial_{\nu} v_n = 0,\quad  \text{ on } \partial \Omega \times (0,T),
\\
u_n(x,\cdot)-a(x)\in H_{\alpha}(0,T),\ 
v_n(x,\cdot)-b(x)\in H_{\alpha}(0,T),\quad \mbox{ for almost all }
x\in \OOO.
\end{array}\right.
\end{equation}
Here we recall that $-Aw = \Delta w - p_0w$, $Q= \OOO\times (0,T)$, and for $n\in \N$ we set
$$
F_n(u_n,v_n) := f_n(u_n,v_n)+p_0u_n ,\
G_n(u_n,v_n) := g_n(u_n,v_n)+p_0v_n.
$$

Now we can prove:
\begin{lemma}\label{Lm:exist_trunc}
Let $a, b \in H^1(\OOO) \cap L^{\infty}(\OOO)$ and $a, b \ge 0$ 
in $\OOO$.
\\
(i) For $n\in \N$, there exists a unique solution $(u_n,v_n) 
\in L^2(0,T;H^2(\OOO))^2$ such that $u_n-a, v_n-b 
\in H_{\alpha}(0,T;L^2(\OOO))$.
\\
(ii) $u_n, v_n \ge 0$ in $Q$ for each $n\in \N$.
\end{lemma}
{\it Proof of Lemma \ref{Lm:exist_trunc}.}

Since \eqref{Eq:fngn_norm_infty} implies the global Lipschitz contintuity of $f_n, g_n$ over 
$\R^2$ for each $n$, in terms of $a, b \in H^1(\OOO) \cap L^{\infty}(\OOO)$, 
we can apply a usual iteration method to establish 
the unique existence of $(u_n,v_n) \in L^2(0,T;H^2(\OOO))^2$ 
and $u_n-a, v_n-b \in H_{\alpha}(0,T;L^2(\OOO))$, which complete the proof
of Lemma \ref{Lm:exist_trunc} (i).
Further details can be found, for example, in \cite{FLY,LY1,LY2}.

Next we will prove (ii).  First we can find functions $C^1_n(x,t),C^2_n(x,t)\in L^{\infty}(Q)$ such that 
\begin{equation}\label{Eq:bdd_fngn}
\begin{aligned}
&\vert  f_n(u_n(x,t), v_n(x,t))\vert  \le C^1_n(x,t) u_n(x,t),  \\
&\vert  g_n(u_n(x,t), v_n(x,t))\vert  \le C^1_n(x,t) v_n(x,t), \quad \mbox{ for almost all } (x,t) \in Q.
\end{aligned}                   
\end{equation}
Actually, we can set $C^1_n(x,t)=\operatorname{sgn}(u_n(x,t)) \|\partial_1 f_n\|_{L^{\infty}(Q)}$ and $C^2_n(x,t)=\operatorname{sgn}(v_n(x,t)) \|\partial_2 f_n\|_{L^{\infty}(Q)}$, where
$$
\operatorname{sgn}(\xi):=
\left\{\begin{array}{r}
 1, \quad \mbox{ if } |\xi|\geq 0, \\ 
-1, \quad \mbox{ if }  |\xi|<0.
\end{array}\right.
$$
Indeed, by \eqref{Eq:fngn_norm_infty} and \eqref{Eq:fngn_00}, we apply the mean value theorem and obtain
\begin{align*}
& \vert f_n(u_n(x,t), v_n(x,t)) \vert 
= \vert f_n(u_n(x,t), v_n(x,t)) - f_n(0, v_n(x,t))\vert\\
\le& \Vert \ppp_1f_n\Vert_{L^{\infty}(\R^2)}\vert u_n(x,t)\vert.
\end{align*}
The proof for $g_n$ is similar and so the verification of \eqref{Eq:bdd_fngn} is 
complete.

Then we can rewrite \eqref{Eq:nonlin_system_appro} as
\begin{equation}
\left\{\begin{array}{l}
\partial_t^{\alpha} (u_n-a) = \Delta u_n - C^1_n(x,t)u_n(x,t)
+ W_n(x,t),\quad 
(x,t) \in Q, \\ 
\partial_t^{\alpha}(v_n-b) = \Delta v_n - C_n^2(x,t)v_n(x,t) + V_n(x,t), \quad 
(x,t) \in Q, \\
\partial_{\nu} u_n = \partial_{\nu} v_n = 0, \quad \text{ on } 
\partial\Omega\times (0,T), \\
u_n(x,\cdot)-a(x)\in H_{\alpha}(0,T), \quad 
v_n(x,\cdot)-b(x)\in H_{\alpha}(0,T), 
\end{array}\right.
\end{equation}
where we set 
\begin{align*}
W_n(x,t) := C_n^1(x,t)u_n(x,t) + f_n(u_n(x,t), v_n(x,t)),\\
V_n(x,t) := C_n^2(x,t)v_n(x,t) + g_n(u_n(x,t), v_n(x,t)).
\end{align*}
We note that $W_n,V_n \ge 0$ in $Q$ by \eqref{Eq:bdd_fngn}.  
Therefore we apply Lemma \ref{Lm:non_nega}  
to $u_n$ and $v_n$, so that $u_n, v_n \ge 0$ in $Q$.
Thus the proof of Lemma \ref{Lm:exist_trunc} is complete.
$\blacksquare$

Next, integrating the first two equations in \eqref{Eq:nonlin_system_appro}
with respect to $x$, we obtain
\begin{equation}\label{Eq:nonlin_system_inte_appro}
\begin{aligned}
\partial^{\alpha}_t \int_{\Omega} (u_n-a)(x,t) \mathrm{d}x 
&= \int_{\Omega}f_n(u_n(x,t),v_n(x,t)) \mathrm{d}x,\\
\partial^{\alpha}_t \int_{\Omega} (v_n-b)(x,t) \mathrm{d}x 
&= \int_{\Omega}g_n(u_n(x,t),v_n(x,t))  \mathrm{d}x.
\end{aligned}
\end{equation}
Here we use the fact that $\int_{\Omega} \Delta u_n(x,t) \mathrm{d}x
= \int_{\Omega} \Delta v_n(x,t) \mathrm{d}x=0$ because 
$\partial_{\nu}u_n=\partial_{\nu}v_n=0$ on $\partial\Omega \times (0,T)$.
Summing up the above equations in $u_n$ and $\lambda v_n$, we can obtain
\begin{equation}
\partial_t^{\alpha} \int_{\Omega} (u_n-a+\lambda v_n- \lambda b)(x,t) \mathrm{d}x=\int_{\Omega} f_n(u_n,v_n)+\lambda g_n(u_n,v_n) \mathrm{d}x \leq 0.
\end{equation}

By applying the fractional integral to both sides, by 
$u_n, v_n \ge 0$ in $Q$, we obtain
\begin{equation}\label{Eq:norm_est_un_vn}
\|u_n(\cdot, t)\|_{L^1(\Omega)} + \|v_n(\cdot, t)\|_{L^1(\Omega)} 
\leq C( \|a\|_{L^1(\Omega)}+ \|b\|_{L^1(\Omega)}).
\end{equation}
Integrating each equation of \eqref{Eq:nonlin_system_appro} with respect 
to $x$ and $t$ and using $\ppp_{\nu}u_n = \ppp_{\nu}v_n = 0$ 
on $\ppp\OOO\times (0,T)$, we derive
\begin{equation}\label{Eq:norm_est_unvn2}
\begin{aligned}
\|f_n(u_n,v_n)\|_{L^1(0,T;L^1(\Omega))} = \|J^{1-\alpha}(u_n-a)\|
_{L^1(0,T;L^1(\Omega))},\\
\|g_n(u_n,v_n)\|_{L^1(0,T;L^1(\Omega))} = \|J^{1-\alpha}(v_n-b)\|_{L^1(0,T;L^1(\Omega))}.
\end{aligned}
\end{equation}
Since $\Vert J^{1-\alpha}w(x,\cdot)\Vert_{L^1(0,T)}
\le C\Vert w(x,\cdot)\Vert_{L^1(0,T)}$ for $w \in L^1(0,T;L^1(\OOO))$ by Young's inequality, we see 
\begin{equation}\label{Eq:fngn_norm_L1}
\begin{aligned}
&\|J^{1-\alpha}(u_n-a)\|_{L^1(0,T;L^1(\Omega))}
+ \|J^{1-\alpha}(v_n-b)\|_{L^1(0,T;L^1(\Omega))} \\
\le &C(\Vert u_n-a\|_{L^1(0,T;L^1(\Omega))}
+ \Vert v_n-b\|_{L^1(0,T;L^1(\Omega))}).
\end{aligned}
\end{equation}
Therefore, in terms of \eqref{Eq:norm_est_un_vn} - \eqref{Eq:fngn_norm_L1}, we reach 
\begin{equation}\label{Eq:norm_est_appro}
\begin{aligned}
&\|u_n(\cdot, t)\|_{L^1(\Omega)}+\|v_n(\cdot, t)\|_{L^1(\Omega)} 
+ \|f_n(u_n, v_n)\|_{L^1(0,T;L^1(\Omega))}+ \|g_n(u_n, v_n)\|
_{L^1(0,T;L^1(\Omega))} \\
\le &C(\Vert a\Vert_{L^1(\OOO)} + \Vert b\Vert_{L^1(\OOO)}),
\end{aligned}
\end{equation}
for all $n\in \mathbb{N}$.
We emphasize that the constant $C>0$ is independent of $n\in \N$.
By \eqref{Eq:norm_est_appro}, we can obtain
\begin{equation}\label{Eq:L1normFG}
\Vert f_n(u_n,v_n)\Vert_{L^1(Q)} + \Vert g_n(u_n,v_n)\Vert_{L^1(Q)} \le C
\end{equation}
for all $n \in \N$.

Next we will establish the convergence of $f_n(u_n,v_n)$ and $g_n(u_n,v_n)$.
We set $a_0:= \|a+\lambda b\|_{L^{\infty}(\Omega)}$ and
$W:= u_n + \la v_n$ in $Q$.  Hence \eqref{Eq:nonlin_system_appro} implies
\begin{equation}\label{Eq:fn_esti_1}
\left\{\begin{array}{l}
\partial_t^{\alpha} (W-a- \lambda b) = \Delta W 
+ (f_n(u_n,v_n)+\la g_n(u_n,v_n)), \quad \mbox{ in }Q,  \\ 
\partial_{\nu}W = 0, \quad \mbox{ on }\partial\Omega \times (0,T).
\end{array}\right.
\end{equation}
By the definition of $f_n, g_n$, we have
\begin{equation}\label{Eq:fngn_comp}
f_n(\xxee) + \la g_n(\xxee)
= \psi_n(\xxee)(f(\xxee) + \la g(\xxee)) \le 0 \quad 
\mbox{ for all } \xi, \eta \in \R. 
\end{equation}
On the other hand, $\www{W}(x,t) := \Vert a+\la b\Vert_{L^{\infty}(\OOO)}$
satisfies 
$$
\left\{ \begin{array}{l}
 \pppa (\www{W} - \Vert a+\la b\Vert_{L^{\infty}(\OOO)})
= \Delta\www{W}, \quad \mbox{ in }Q, \\
 \ppp_{\nu}\www{W} = 0, \quad \mbox{ on } \ppp\OOO \times (0,T).
\end{array}\right.
$$
Applying the comparison principle\cite[(Corollary 1)]{LY1} to $W$ and 
$\www{W}$, in terms of \eqref{Eq:fngn_comp}, we obtain
$$
W(x,t) \le \www{W}(x,t), \quad (x,t) \in Q.
$$
By Lemma \ref{Lm:exist_trunc} (ii), we have $W(x,t) = u_n(x,t) + \la v_n(x,t) \ge 0$ for
$(x,t) \in Q$.  Therefore,
$$
0 \le u_n(x,t) + \la v_n(x,t) \le \Vert a + \la b\Vert_{L^\infty(\OOO)},
\quad (x,t) \in Q.
$$
Since $\xi + \eta \le \max\left\{1, \, \frac{1}{\la}\right\}
(\xi + \la \eta)$ for $\xi, \eta \ge 0$, then we have
\begin{equation}\label{Eq:est_un_vn}
\begin{aligned}
0 &\leq u_n+v_n \le \max\left\{ 1, \, \frac{1}{\la}\right\}(u_n+\la v_n)
\\
&\leq C\Vert a + \la b\Vert_{L^\infty(\OOO)}
=: C_1, \quad \mbox{ almost everywhere in }Q.
\end{aligned}
\end{equation}
Because $u_n,v_n$ are both non-negative, we deduce $u_n,v_n \leq C_1$. Therefore
\begin{equation}\label{Eq:fngn_bdd_unif}
\begin{aligned}
& \vert f_n(u_n(x,t), v_n(x,t)) \vert 
 + \vert g_n(u_n(x,t), v_n(x,t)) \vert 
 \\
\le& \sup_{\vert \xi\vert, \vert \eta\vert \le C_1} \vert f_n(\xi,\eta)\vert
+ \sup_{\vert \xi\vert, \vert \eta\vert \le C_1} \vert g_n(\xi,\eta)\vert
\\
\le& \sup_{\vert \xi\vert, \vert \eta\vert \le C_1} \vert f(\xi,\eta)\vert
+ \sup_{\vert \xi\vert, \vert \eta\vert \le C_1} \vert g(\xi,\eta)\vert
=: C_2,
\end{aligned}
\end{equation}
for all $n \in \N$ and almost all $(x,t) \in Q$.  We note that $C_2$ does not depend on $n$.

Then we can choose a subsequence $n^{\prime} \in \N$ such that 
\begin{equation}\label{Eq:fn_f_point}
\begin{aligned}
\lim_{n^{\prime}\to \infty} f_{n^{\prime}}(u_{n^{\prime}}(x,t), \,v_{n^{\prime}}(x,t)) = f(u(x,t),\, v(x,t)),
\quad \mbox{ for almost all } (x,t) \in Q, \\
\lim_{n^{\prime}\to \infty} g_{n^{\prime}}(u_{n^{\prime}}(x,t), \,v_{n^{\prime}}(x,t)) = g(u(x,t),\, v(x,t)),
\quad \mbox{ for almost all }(x,t) \in Q.
\end{aligned}
\end{equation}
In order to prove this, we can assume that $n \ge \frac{C_1}{2}$.
Then, by the definition of $f_n$ and $g_n$, we deduce
\begin{equation}\label{Eq:fngn_lim}
\begin{aligned}
f_n(u_n(x,t), \,v_n(x,t)) = f(u_n(x,t), \,v_n(x,t)), 
\\
g_n(u_n(x,t), \,v_n(x,t)) = g(u_n(x,t), \,v_n(x,t)),
\end{aligned}
\end{equation}
for almost all $(x,t) \in Q:= \Omega \times (0,T)$ if $n\geq \frac{C_1}{2}$. On the other hand, in terms of \eqref{Eq:norm_est_appro} we apply Lemma \ref{Lm:comp} to see that $(u_n,v_n), n\in \mathbb{N}$ contains a convergent subsequence in $L^1(Q)$. Therefore, we can find a subsequence $\{ n^{\prime}\} \subset  \N$ and some $(u,v)\in (L^1(Q))^2$ such that $(u_{n^{\prime}}, v_{n^{\prime}}) \longrightarrow (u,v)$ in $L^1(Q)^2$ as 
$n^{\prime}\to \infty$.
Hence, choosing further a subsequence of $\{n^{\prime}\}$, denoted by the same 
notations, such that
\begin{equation}\label{Eq:un_vn_conv}
\begin{aligned}
&\lim_{n^{\prime} \to \infty} u_{n^{\prime}}(x,t) = u(x,t), \\
&\lim_{n^{\prime} \to \infty} v_{n^{\prime}}(x,t) = v(x,t), \quad
\mbox{ for almost all }(x,t) \in Q.
\end{aligned}
\end{equation}
Hence, by $f,g \in C(\R^2)$ and \eqref{Eq:fngn_lim}-\eqref{Eq:un_vn_conv}, we verified \eqref{Eq:fn_f_point}.

Finally considering \eqref{Eq:fngn_bdd_unif} and \eqref{Eq:fn_f_point}, we apply the Lebesgue convergence
theorem to verify 
$$
\begin{aligned}
\lim_{n\to \infty}\Vert f_n(u_n,v_n) - f(u,v)\Vert_{L^1(Q)} = 0, \\
\lim_{n\to \infty}\Vert g_n(u_n,v_n) - g(u,v)\Vert_{L^1(Q)} = 0.
\end{aligned}
$$

Multiplying both sides of  \eqref{Eq:nonlin_system_appro} with $\psi$, after integrating, we obtain
\begin{equation}\label{Equ:nonlin_system_appro_distr}
\begin{aligned}
(u_n-a, (\partial_t^{\alpha})^{*} \psi) &= (u_n, -A \psi) + (F_n, \psi), \\
(v_n-b, (\partial_t^{\alpha})^{*} \psi) &= (v_n,-A \psi) + (G_n, \psi),
\end{aligned}
\end{equation}
for all $\psi \in C^{\infty} (\ooo{Q})$ satisfying  
 $\partial_{\nu} \psi =0 $ in $\partial\Omega \times (0,T)$ and $\psi(\cdot, T)=0$ in $\Omega$.
Passing $n$ in \eqref{Equ:nonlin_system_appro_distr} to the limit, we can see
\begin{equation}
\begin{aligned}
(u-a, (\partial_t^{\alpha})^{*} \psi) &= (u,\Delta \psi) + (f(u,v), \psi), \\
(v-b, (\partial_t^{\alpha})^{*} \psi) &= (v,\Delta \psi) + (g(u,v), \psi),
\end{aligned}
\end{equation}
for all $\psi \in C^{\infty}(\ooo{Q})$ 
satisfying $\partial_{\nu} \psi =0 $ in $\partial\Omega \times (0,T)$ 
and $\psi(\cdot, T)=0$ in $\Omega$.
This implies that $(u,v)$ is a weak solution to equation \eqref{Eq:nonlin_system_formu}, we have completed the proof of Theorem \ref{Th:well-pose}.
\end{proof}



\section{Blow up for a system 
for diffusion equation with convex nonlinear terms} \label{Sec:blow}
\setcounter{section}{5} \setcounter{equation}{0} 
In this section, we show that the global existence does not hold true
for some nonlinear terms, which is our second main result. Here, we do not require the coefficients in front of the spatial diffusion terms to be equal. Specifically, for the second line in \eqref{Eq:nonlin_system}, we consider $$
\partial_t^{\alpha}(v-b)=d \Delta v + g(u,v).
$$
For the statement,
we introduce the following assumption for the nonlinear terms 
$f$ and $g$.

\begin{assumption}\label{Ass:convex}
For $f,g\in C^{\infty}(\mathbb{R}^2)$, there exist constants $\lambda>0, p>1$ and $C=C_{p,\lambda}>0$ such that
\begin{equation}\label{Eq:conve_assum}
f(\xi,\eta)+\lambda g(\xi,\eta)\geq C(\xi^p+\eta^p),\quad \text{ for all } 
\xi, \eta \geq 0.
\end{equation}
\end{assumption}

Now we are ready to state our main result on the blow-up with upper bound 
of the blow-up time for $p>1$.
\begin{theorem}\label{Th:blow-up}
Let $f,g$ satisfy the Assumption \ref{Ass:convex} and $a, b\in H^{2\gamma}(\Omega)$, $\frac{3}{4}<\gamma\leq 1$, satisfy $\partial_{\nu}a=\partial_{\nu}b=0$ on $\partial\Omega$ and $a, b\geq 0,\not\equiv 0$ in $\Omega$. Then there exists some $T=T_{\alpha,p,a,b}>0$ such that the solutions of 
\eqref{Eq:nonlin_system_formu} satisfying 
$$
u, v\in C([0,T];H^2(\Omega)),\quad u-a, v-b \in H_{\alpha}(0,T;L^2(\Omega)),
$$
exist for $0<t<T_{\alpha,p,a,b}$ and $$
\lim_{t \uparrow T_{\alpha,p,a,b}} \left(\|u(\cdot, t)\|_{L^1(\Omega)}+\|v(\cdot, t)\|_{L^1(\Omega)}  \right)=\infty
$$
holds. Moreover, we can bound $T_{\alpha,p,a,b}$ from above as 
$$
T_{\alpha,p,a,b}\leq \left( \frac{1}{(p-1) \Gamma(2-\alpha)C_{p,\lambda}^{-1}\left(\frac{1}{|\Omega|} \int_{\Omega} a +\lambda b \mathrm{~d} x\right)^{p-1}} \right)^{\frac{1}{\alpha}}=:T^*(\alpha,p,a,b).
$$
\end{theorem}

\begin{proof}
\textbf{Step 1.}
Our proof is similar to the one for a single equation
(\cite{FLY}).  We show the following two lemmas from \cite{FLY}.

\begin{lemma}
Let $h\in L^2(0,T)$ and $c\in C[0,T]$. Then there exists a unique solution $y\in H_{\alpha}(0,T)$ to
$$
\partial_t^{\alpha} y - c(t)y=h,\quad 0<t<T.
$$
Moreover, if $h\geq 0$ in $(0,T)$, then $y\geq 0$ in $(0,T)$.
\end{lemma}

\begin{lemma}\label{Le:compa}
Let $c_0>0, a_0\geq 0, p>1$ be constants and $y-a_0,z-a_0\in H_{\alpha}(0,T)\cap C[0,T]$ satisfy
$$
\partial_t^{\alpha}(y-a_0)\geq c_0y^p,\quad \partial_t^{\alpha}(z-a_0)\leq c_0z^p,\quad \text{ in } (0,T).
$$
Then $y\geq z$ in $(0,T)$.
\end{lemma}

\textbf{Step 2.}
We set 
$$
\theta(t):=\int_{\Omega}\left( u(x,t)+\lambda v(x,t) \right)\mathrm{d}x=\int_{\Omega}\left( (u-a)+\lambda (v-b) \right)(x,t)\mathrm{d}x+m_0
$$
where $m_0:=\int_{\Omega}(a+\lambda b)\mathrm{d}x$. Here we see that $m_0>0$ because $a, b\geq 0,\not\equiv 0$ in $\Omega$ and $\lambda>0$ by the assumption of Theorem \ref{Th:blow-up}. Then by \eqref{Eq:nonlin_system} we have
$$
\begin{aligned}
\partial_t^{\alpha}(\theta(t)-m_0) &= \int_{\Omega}(\Delta u +\lambda d \Delta v)(x,t) \mathrm{d}x + \int_{\Omega}(f(u,v) +\lambda g(u,v))(x,t) \mathrm{d}x \\
&= \int_{\Omega}(f(u,v) +\lambda g(u,v))(x,t) \mathrm{d}x.
\end{aligned}
$$
On the other hand, we have
$$
\theta(t)=\int_{\Omega} (u+\lambda v)\mathrm{d}x \leq \left( \int_{\Omega} \mathrm{d}x \right)^{\frac{1}{p^{\prime}}}  \left( \int_{\Omega} (u+\lambda v)^p \mathrm{d}x \right)^{\frac{1}{p}}=|\Omega|^{\frac{1}{p^{\prime}}} \left( \int_{\Omega} (u+\lambda v)^p \mathrm{d}x \right)^{\frac{1}{p}}
$$
with $\frac{1}{p^{\prime}}+\frac{1}{p}=1$, and $(u+\lambda v)^p\leq C_{p,\lambda}(u^p+v^p)$ for $u,v\geq 0$. Hence, \eqref{Eq:conve_assum} yields
$$
\begin{aligned}
  \theta^p(t)&\leq |\Omega|^{\frac{p}{p^{\prime}}} \int_{\Omega} (u+\lambda v)^p \mathrm{d}x\leq C_{p,\lambda} |\Omega|^{\frac{p}{p^{\prime}}} \int_{\Omega} (u^p+v^p)\mathrm{d}x\\
&\leq C_{p,\lambda} |\Omega|^{\frac{p}{p^{\prime}}}\int_{\Omega}(f(u,v) +\lambda g(u,v))(x,t) \mathrm{d}x.
\end{aligned}
$$
So we obtain
$$
\partial_t^{\alpha}(\theta(t)-m_0)
\geq C_{p,\lambda}^{-1} |\Omega|^{-\frac{p}{p^{\prime}}} \theta^p(t),
\quad 0<t<T.
$$

\textbf{Step 3.} 
This step is devoted to the construction of a lower solution 
$\underline{\theta}(t)$ satisfying
\begin{equation}\label{Eq:lower_eta}
\partial_t^{\alpha}(\underline{\theta}(t)-m_0)\leq C_0 \underline{\theta}^p(t),
\quad 0<t<T-\epsilon,\quad \lim_{t \uparrow T} \underline{\theta}(t)=\infty.
\end{equation}
Here we set $C_0:=C_{p,\lambda}^{-1} |\Omega|^{-\frac{p}{p^{\prime}}}$.
The construction of $\underline{\theta}(t)$ is same as the proof in \cite{FLY}.
As a possible lower solution, we consider
$$
\underline{\theta}(t):=m_0 \left( \frac{T}{T-t} \right)^m,\ m\in \mathbb{N}.
$$
By the definition $\partial_t^{\alpha}(\underline{\theta}(t)-a_0)(t)
=\mathrm{d}^{\alpha}_t \underline{\theta}(t)$. We can get 
$$
\partial_t^{\alpha}(\underline{\theta}(t)-m_0)(t)
=\mathrm{d}^{\alpha}_t \underline{\theta}(t)=m_0T^m \mathrm{d}^{\alpha}_t \left( \frac{1}{(T-t)^m} \right) \leq \frac{m_0 T^{m+1-\alpha}m}{\Gamma(2-\alpha)} \frac{1}{(T-t)^{m+1}}.
$$
This equation holds for arbitrary $m\in \mathbb{N}, T>0$ and $0<t<T-\epsilon$.

Finally we claim that for any $p>1$ and $m_0>0$, there exist constants $m\in \mathbb{N}$ and $T>0$ such that
\begin{equation}\label{Eq:est_eta}
\frac{m_0T^{m+1-\alpha}m}{\Gamma(2-\alpha)} \frac{1}{(T-t)^{m+1}}
\leq C_0 \underline{\theta}^p(t)=\frac{C_0 m_0^pT^{mp}}{(T-t)^{mp}},\ 0<t<T-\epsilon.
\end{equation}
Therefore, if
\begin{equation}
\begin{aligned} T & \geq\left(\frac{m}{\Gamma(2-\alpha) C_{0} m_{0}^{p-1}}\right)^{\frac{1}{\alpha}} \geq\left(\frac{1}{(p-1) \Gamma(2-\alpha) C_{0} m_{0}^{p-1}}\right)^{\frac{1}{\alpha}} \\ & =\left\{(p-1) \Gamma(2-\alpha)C_{p,\lambda}^{-1}\left(\frac{1}{|\Omega|} \int_{\Omega} (a +\lambda b) \mathrm{~d} x
\right)^{p-1}\right\}^{-\frac{1}{\alpha}}
=: T^{*}(\alpha, p, a,b),\end{aligned}
\end{equation}
then \eqref{Eq:est_eta} is satisfied.

Consequently, we can obtain 
$$
\underline{\theta}(t):=m_0 \left( \frac{T^{*}(\alpha, p, a,b)}{T^{*}(\alpha, p, a,b)-t} \right)^m
$$
satisfies \eqref{Eq:lower_eta}.

Now it suffices to apply Lemma \ref{Le:compa} to get
$$
\theta(t)\geq \underline{\theta}(t),\quad 0\leq t\leq T^{*}(\alpha, p, a,b)-\epsilon.
$$
Since $\epsilon>0$ was arbitrarily chosen, we obtain
$$
\int_{\Omega}\left( u(x,t)+\lambda v(x,t) \right)\mathrm{d}x
= \theta(t)\geq \underline{\theta}(t) 
=  m_0 \left( \frac{T^{*}(\alpha, p, a,b)}{T^{*}(\alpha, p, a,b)-t} \right)^m.
$$
This means that the solution $u+v$ cannot exist for $t>T^{*}(\alpha, p, a,b)$. Hence, the blow-up time $T(\alpha, p, a,b)\leq T^{*}(\alpha, p, a,b)$. The proof of Theorem \ref{Th:blow-up} is complete.
\end{proof}

\section{Concluding remarks} \label{Sec:concl}
\setcounter{section}{6} \setcounter{equation}{0} 

It's worth mentioning that the conditions on $f$ and $g$ in the Assumption \ref{Ass:nonlin_term} are $f(0,\eta)= g(\xi,0)= 0$ for all $\xi\geq0$ and $\eta \geq 0$, and $f, g$ are local Lipschitz continuous. If we improve the regularity of $f$ and $g$ such that $f,g\in C^1(\mathbb{R}^2)$, we can obtain the non-negativity by comparison principle.
For simplicity, we choose an initial value $a=0$, and denote 
$\widetilde{u}$ as the solution
of
$$
\partial_t^{\alpha} \widetilde{u} 
= \Delta \widetilde{u} + f(\widetilde{u},\widetilde{v}) 
- f(0,\widetilde{v})
$$
Noticing $f(\widetilde{u},\widetilde{v}) - f(0,\widetilde{v}) 
= \partial_1 f(\xi,\widetilde{v})\widetilde{u}$, we can rewrite the above 
equation as 
$$
\partial_t^{\alpha} \widetilde{u} = \Delta \widetilde{u} 
+ \partial_1 f(\xi,\widetilde{v})\widetilde{u}.
$$
Here $\xi$ is some number between $0$ and $\widetilde{u}(x,t)$.
Similarly, we obtain the equation for $\widetilde{v}$ as
$$
\partial_t^{\alpha} \widetilde{v} = \Delta \widetilde{v} 
+ \partial_2 g(\widetilde{u},\eta)\widetilde{v},
$$
where $\eta$ is some number between $0$ and $\widetilde{v}$.
Therefore, by the comparison principle\cite{LY2} and Lemma \ref{Lm:non_nega}, 
we can derive the non-negativity of the solutions $\widetilde{u}$ and $\widetilde{v}$.

Moreover, it can be observed that limited in $L^1(\Omega\times (0,T))$, 
we can derive the global existence of solutions. For the uniqueness of solutions, we need to impose stronger priori assumptions on the right-hand side term and the initial values. Similarly to \cite{LY2}, we impose the following requirements on the nonlinear term $g(u,v)$. For $g:\mathcal{D}(A_0^{\gamma} \times A_0^{\gamma})\to L^2(\Omega)$, we assume that for some constant $m>0$, there exists a constant $C_g=C_g(m)>0$ such that
$$
\left\{\begin{array}{l}\text { (i) }
\|g(u,\cdot)\| + \|g(\cdot,v)\| \leq C_{g},\quad \left\|g\left(u_{1},\cdot \right)-g\left(u_{2},\cdot \right)\right\| \leq C_{g}\left\|u_{1}-u_{2}\right\|_{\mathcal{D}\left(A_{0}^{\gamma}\right)},
\\ \left\|g\left(\cdot,v_{1}\right)-g\left(\cdot,v_{2}\right)\right\| \leq C_{g}\left\|v_{1}-v_{2}\right\|_{\mathcal{D}\left(A_{0}^{\gamma}\right)}
\\ \text { if }\|u\|_{\mathcal{D}\left(A_{0}^{\gamma}\right)},\|v\|_{\mathcal{D}\left(A_{0}^{\gamma}\right)},\left\|u_{1}\right\|_{\mathcal{D}\left(A_{0}^{\gamma}\right)},\left\|u_{2}\right\|_{\mathcal{D}\left(A_{0}^{\gamma}\right)},\left\|v_{1}\right\|_{\mathcal{D}\left(A_{0}^{\gamma}\right)},\left\|v_{2}\right\|_{\mathcal{D}\left(A_{0}^{\gamma}\right)} \leq m 
\\ \text { and } 
\\ \text { (ii) there exists a constant } \varepsilon \in\left(0, \frac{3}{4}\right) \text { such that } 
\\ \|g(u,\cdot)\|_{H^{2 \varepsilon}(\Omega)} + \|g(\cdot,v)\|_{H^{2 \varepsilon}(\Omega)} \leq C_{g}(m) \quad \text { if }\|u\|_{\mathcal{D}\left(A_{0}^{\gamma}\right)},\|v\|_{\mathcal{D}\left(A_{0}^{\gamma}\right)} \leq m,
\end{array}\right.
$$
along with initial values $a,b$ satisfying $\|a\|_{\mathcal{D}\left(A_{0}^{\gamma}\right)}, \|b\|_{\mathcal{D}\left(A_{0}^{\gamma}\right)}\leq m$.  After expressing solutions $u(t)$ and $v(t)$ as
$$
\begin{aligned}
u(t)=S(t)a+\int_0^t K(t-s) \left(p_0 u -g(u(s),v(s))\right) \mathrm{d}s, \\
v(t)=S(t)b+\int_0^t K(t-s) \left(p_0 v + g(u(s),v(s))\right) \mathrm{d}s,
\end{aligned}
$$
we can apply the contraction theorem similarly to the approach demonstrated in \cite{LY2}. In this situation, there exists a constant $T=T(m)>0$ such that the solutions to the system \eqref{Eq:nonlin_system_formu} in $L^2(0,T;H^2(\Omega)) \cap C([0,T];\mathcal{D}(A^{\gamma}_0))$ are unique.

\section{Appendix. Proof of Lemma \ref{Lm:non_nega}} 
\setcounter{section}{7} \setcounter{equation}{0} 

Since $1 < \kappa < \infty$, the embedding $L^\infty(Q) \hookrightarrow L^\kappa(Q)$ holds, thus we have $c \in L^\kappa(Q)$. Furthermore, given that $C_0^\infty(Q)$ is dense in $L^\kappa(Q)$, there exist functions $\{c_n\} \in C_0^\infty(Q)$ such that $\lim_{n \to \infty} c_n = c$ in $L^\kappa(Q)$.

We define the operator $A_n$ regarding $c_n(x,t)$ in $L^2(\Omega)$ similar to \cite{LY1} as
$$
(A_n v)(x)=- \Delta v(x)-c_n(x,t) v(x), \quad x \in \Omega.
$$
Correspondingly, the operators $S(t)$ and $K(t)$ related to $\left(-A_{0} v\right)(x):= \Delta v(x)-c_{0} v(x)$ with an arbitrary constant $c_0>0$ remain the same
$$
S(t) a=\sum_{n=1}^{\infty} E_{\alpha, 1}\left(-\mu_{n} t^{\alpha}\right)\left(a, \varphi_{n}\right) \varphi_{n}, \quad a \in L^{2}(\Omega), t>0, \\
$$
and
$$
 K(t) a=\sum_{n=1}^{\infty} t^{\alpha-1} E_{\alpha, \alpha}\left(-\mu_{n} t^{\alpha}\right)\left(a, \varphi_{n}\right) \varphi_{n}, \quad a \in L^{2}(\Omega), t>0.
$$
Additionally, the operators in \cite[(2.14)]{LY1} are updated to
$$
\begin{array}{l}
G(t):=\int_{0}^{t} K(t-s) F(s) d s+S(t) a,
\\ 
Q_n v(t)=Q_n(t) v(t):=\left(c_{0}+c_n(\cdot, t)\right) v(t).
\end{array}
$$
In accordance on the previous modifications, the solution presented in \cite[(4.3)]{LY1} has been updated to reflect the changes in the operators $A_n$ and $Q_n$ as defined above. The updated solution is given by
$$
u_n(F,a)(t)=G(t)+\int_0^t K(t-s)Q_n u_n(s)\mathrm{d}s,\quad 0<t<T.
$$
Since $c_n\in C_0^{\infty}(Q)$, according to \cite[Theorem 2]{LY1}, we know that $u_n(F,a)(x,t)\geq 0$ in $\Omega \times (0,T)$ and that $u_n(F,a) \in L^2(0,T;H^2(\Omega))$ as well as $u_n(F,a) - a \in H_{\alpha}(0,T;L^2(\Omega))$.
Following a similar approach as in \cite{LY1}, we can extract a subsequence ${u_{n^{\prime}}(F,a)}$ from ${u_{n}(F,a)}$ such that ${u_{n^{\prime}}(F,a)} \to {u(F,a)}$ in $L^2(0,T;H^2(\Omega))$ and ${u_{n^{\prime}}(F,a)} - a \to {u(F,a)} - a$ in $H_{\alpha}(0,T;L^2(\Omega))$. Finally, $u_n(F,a)(x,t)\geq 0$ implies that $u(F,a)\geq 0$ in $Q$.




\begin{acknowledgements}
D. Feng is supported by Key-Area Research and Development Program of Guangdong Province (No.2021B0101190003) and Science and Technology Commission of Shanghai Municipality (23JC1400501). M. Yamamoto is supported partly by Grant-in-Aid for 
Scientific Research (A) 20H00117 
and Grant-in-Aid for Challenging Research (Pioneering) 21K18142 of 
Japan Society for the Promotion of Science.
 \end{acknowledgements}

 \section*{\small
 Conflict of interest} 

 {\small
 The authors declare that they have no conflict of interest.}




\bigskip  

\small 
\noindent
{\bf Publisher's Note}
Springer Nature remains neutral with regard to jurisdictional claims in published maps and institutional affiliations.

\end{document}